\documentclass{article}
\title{Non-deterministic inductive definitions}
\author{Benno van den Berg\footnote{Mathematisch Instituut, Universiteit Utrecht,
PO. Box 80010, 3508 TA Utrecht. Email address: B.vandenBerg1@uu.nl. Supported by the Netherlands Organisation for Scientific Research (NWO).
}}
\date{August 10, 2012}
\usepackage{ast2}
\usepackage{proof}
\begin{document}

\maketitle

\begin{abstract}
\noindent
We study a new proof principle in the context of constructive Zermelo-Fraenkel set theory based on what we will call ``non-deterministic inductive definitions''. We give applications to formal topology as well as a predicative justification of this principle.
\end{abstract}

\section{Introduction}

There is a distinctive foundational stance that has sometimes been called ``generalised predicativity'', which is characterised by a rejection of impredicative definitions combined with an acceptance of a wide variety of inductively defined sets. The system which expresses this philosophy in its purest form is Martin-L\"of's type theory \cite{martinlof84}.

Martin-L\"of's type theory is intended to be an open-ended framework and, as a consequence, comes in different versions of varying strength. For the purposes of this paper, the relevant system is the theory {\bf ML$_{1W}$V} from \cite{grifforrathjen94} and \cite{rathjentupailo06}, with one iterative universe closed under W-types and with the ``extensional'' rules for the identity types, as in \cite{martinlof84}. (We believe the arguments in this paper still go through if one works with the intensional identity types and the axiom of function extensionality. It is unclear to us what happens when we drop function extensionality.) Occasionally, we will also consider the system {\bf ML$_{1}$V}, where we drop the requirement that the iterative universe it is closed under W-types.

Working in {\bf ML$_{1W}$V} is hard, however, for several reasons: the syntax is unfamiliar and complicated, and it lacks extensional constructs (like quotient types). For this reason, people have sought systems which are easier to work with, but can still be interpreted in type theory. The most prominent among them is Peter Aczel's constructive set theory {\bf CZF} \cite{aczel78,aczelrathjen01}.

But constructive set theory has some difficulties in catching up on the type theory, for there are some mathematical results, especially in what is called formal topology, which can be proved in {\bf ML$_{1W}$V}, but not in {\bf CZF}. Formal topology is a particular approach to the subject of topology, which, by taking the notion of a basis as the starting point, can be developed in a way which is acceptable from the generalised predicative point of view. It should be one of features of {\bf CZF} that it can act as a set-theoretic foundation for formal topology, but it does not quite live up to that.

The first problem one encounters is that formal topology makes heavy use of inductive definitions,  but {\bf CZF} is unable to prove that these generate sets. This was remedied early on by Peter Aczel, when he introduced the Regular Extension Axiom {\bf REA} \cite{aczel86}. {\bf CZF} + {\bf REA} can still be interpreted in {\bf ML$_{1W}$V}  and allows one to prove the existence of a wide variety of inductively generated sets. An alternative solution (the combination of {\bf WS} and {\bf AMC}) was proposed by the author together with Ieke Moerdijk \cite{bergmoerdijk11}. Indeed, in both extensions of {\bf CZF} one can prove the Set Compactness Theorem, which allows one to prove that every inductively generated formal space is set-presented, which is an important fact in formal topology.

But, unfortunately, it seems that {\bf CZF} + {\bf REA} and {\bf CZF} + {\bf WS} + {\bf AMC} are still not capable of capturing all the desired results in formal topology. In particular, people have not managed to prove in these systems Palmgren's results from \cite{palmgren05, palmgren06} on points and coequalizers of formal spaces (see Section 5 below). We do not have a proof that this is impossible (although we are inclined to think that it is).

In the meantime, various people have attempted to find set-theoretic principles which would allow one to prove these other results. The subject of this paper is the author's proposal for such a principle, which he has dubbed NID, for ``non-deterministic inductive definitions''. The contention is that the NID principle provides an elegant and relatively simple solution to proving all the additional results in formal topology which go beyond these frameworks. Moreover, the principle is acceptable from a generalised predicative point of view, because it is valid under the type-theoretic interpretation of {\bf CZF} in {\bf ML$_{1W}$V} (this will be proved in Section 6). Another proposal is due to Aczel, Ishihara, Nemoto and Sangu (see \cite{aczeletal12}) and we compare our proposals in Section 7 of this paper.

The contents of this paper are therefore as follows: in Section 2 we formulate the NID principle and in Section 3 we prove some first applications, especially to inductive and coinductive types. To get a more streamlined presentation of applications of NID to formal topology, it turns out to be helpful to have a formulation of the NID principle in terms of logic (an idea of Peter Aczel). This reformulation we will develop in Section 4. In Section 5, then, we proceed to develop the applications of the NID principle to formal topology. After that, in Section 6, we will show that the NID principle is validated by the interpretation of {\bf CZF} in Martin-L\"of's type theory with one inductive universe closed under W-types. This section was inspired by earlier work by Aczel \cite{aczel06} and unpublished notes by Ishihara. After that, in Section 7, we will compare our work with the preprint \cite{aczeletal12}. Finally, we end this paper with formulating several questions which have been left open in this paper.

The author would like to thank Hajime Ishihara for showing him some early drafts of \cite{aczeletal12} and for very useful and enjoyable discussions, also with Erik Palmgren, during the author's stay at the Institut Mittag-Leffler in Fall 2009. He is also grateful to the Institut for awarding him a fellowship, as well as to the referees for useful comments.

\section{The NID principle}

We work in {\bf CZF}, unless expressly indicated otherwise.

\begin{rema}{notation}
Throughout this paper we will call a set $A$ \emph{finite}, if there is a natural number $n \in \NN$ and a surjection $\{1, \ldots, n \} \to A$. Such sets have also been called K-finite (for Kuratowski-finite) or finitely enumerable to distinguish them from other constructive notions of finiteness. An important and useful property of these K-finite sets is that the collection Finpow($A$) of K-finite subsets of a set $A$ can be proved to be a set in {\bf CZF}.

We will write ${\rm Pow}(A)$ for the collection of all subsets of a set $A$. Of course, this cannot be proved to be a set in {\bf CZF}.
\end{rema}

\begin{defi}{rules}
Let $X$ be any set. By a \emph{rule} on $X$, we will mean a pair $(a, b)$ with $a$ and $b$ subsets of $X$. A rule is called \emph{elementary} if $a$ is a singleton and \emph{finitary} if $a$ is finite. If $b$ is a singleton, the rule will be called \emph{deterministic}. A subset $Y \subseteq X$ is \emph{closed} under the rule $(a, b)$, if
\[ a \subseteq Y \Rightarrow b \between Y. \]
Recall that $b \between Y$ is Sambin's notation for: the intersection of $b$ and $Y$ is inhabited. Therefore the intuitive meaning of a rule $(a, b)$ is: if all elements of $a$ belong to the set, then at least one element from $b$ should belong to the set. Finally, if ${\cal R}$ is a set of rules on $X$, we will call a subset $Y \subseteq X$ ${\cal R}$-closed, if it is closed under all rules in ${\cal R}$, and write ${\rm Clos}_{\cal R}(X)$ for the class consisting of all ${\cal R}$-closed subsets of $X$.
\end{defi}

\begin{exam}{primeideals}
An example of a non-deterministic inductive definition is the notion of a prime ideal in a unital, commutative ring. For if $A$ is a commutative ring with 1, then a prime ideal is a subset $P \subseteq A$ which is closed under the following rules:
\begin{displaymath}
\begin{array}{ccccccc}
\infer[r, s \in A]{ \{ r + s \} }{ \{ r, s \} } & &
\infer[r, s \in A]{ \{rs \} }{ \{ s \} } & &
\infer[r, s \in A]{ \{r, s \} }{ \{rs \} } & &
\infer{\emptyset}{ \{ 1 \} }
\end{array}
\end{displaymath}
\end{exam}

\begin{defi}{set-generated}
A subclass $M$ of ${\rm Pow}(X)$, where $X$ is a set, will be called \emph{set-generated}, if there is a subset $G$ of $M$ such that 
\[ (\forall \alpha \in M) \, \alpha = \bigcup \, \{ \beta \in G \, : \, \beta \subseteq \alpha \}, \]
or, equivalently,
\[ (\forall \alpha \in M) \, (\forall a \in \alpha) \, (\exists \beta \in G) \, a \in \beta \subseteq \alpha. \]
\end{defi}

\begin{rema}{setgenframe}
Compare the notion of a set-generated frame: a frame $X$ is set-generated if there is a subset $G \subseteq X$ such that 
\[ x = \bigvee \, \{ g \in G \, : \, g \leq x \} \]
for every $x \in X$.
\end{rema}

The NID principle now reads:
\begin{quote}
{\bf NID principle:} For any set $X$ and set of rules ${\cal R}$ on $X$, the class ${\rm Clos}_{\cal R}(X)$ is set-generated.
\end{quote}
Weaker principles can be obtained by requiring all the rules in ${\cal R}$ to be elementary or finitary: these will be called the elementary and finitary NID principle, respectively. Clearly, NID implies finitary NID implies elementary NID.

\begin{exam}{primeidealsctnd}
Hence finitary NID implies that the class of prime ideals in a commutative ring with unit is set-generated.
\end{exam}

\begin{rema}{NIDprovableinIZF} Note that ${\rm Clos}_{\cal R}(X)$ can be proved to be a set using the Powerset and Full Separation axioms. Therefore the NID principle is provable in {\bf IZF}.
\end{rema}

\section{First applications}

In this section we give some applications of the NID principle. The first is a bit of a curiosity: it says that Fullness is a consequence of elementary NID (over {\bf CZF$^-$}, which is {\bf CZF} minus the Subset collection axiom).

\begin{prop}{elimNIDimpliesfullness}
In {\bf CZF$^-$}, elementary NID implies Fullness.
\end{prop}
\begin{proof}
Recall that Fullness is the statement that for any two sets $A$ and $B$ there is a set $\Sigma$ of total relations from $A$ to $B$ such that every total relation is refined by (i.e., contains) one in this set $\Sigma$. Therefore consider sets $A$ and $B$ and consider the following set of non-deterministic rules on $1 + A + A \times B$:
\begin{displaymath}
\begin{array}{ccc}
\infer[* \in 1, a \in A]{ \{ a \} }{ \{ * \} } & &
\infer[ a \in A ]{ \{ (a, b) \, : b \in B \} }{ \{ a \} }
\end{array}
\end{displaymath}
If ${\cal I}$ generates the closed sets for this elementary non-deterministic inductive definition, then ${\cal J} = \{ I \cap A \times B \, : \, I \in {\cal I}, * \in I \}$ is a full set of total relations from $A$ to $B$. For if $R$ is a total relation from $A$ to $B$, then $1 + A + R$ is a subset of $1 + A + A \times B$ closed under the non-deterministic rules. So there is an element $I \in {\cal I}$ with $* \in I$ and $I \subseteq 1 + A + R$. Then we have for $J = I \cap A \times B$ that $J \subseteq R$ and $J \in {\cal J}$.
\end{proof}

The results below show that elementary NID is especially useful for proving statements related to coinductive types and bisimulation. One problem which elementary NID solves is the following: when one tries to show the consistency of the Anti-Foundation Axiom in {\bf CZF} along the lines of Aczel's book \cite{aczel88}, one has to show that the bisimularity relation is $\Delta_0$ (or that the statement that two graphs are bisimular ``has a small truth-value''). Proving this in {\bf CZF} seems to be hard, if not impossible. But one can readily prove this in {\bf CZF} extended with elementary NID:

\begin{prop}{elementaryNIDandbisimulation}
If $(A, R)$  and $(B, S)$ are two graphs, then elementary NID implies that the class of bisimulations from $(A, R)$ to $(B, S)$ is set-generated. Therefore it also implies that the statement that $(A, R)$ and $(B, S)$ are bisimular is equivalent to a bounded (or $\Delta_0$-)formula.
\end{prop}
\begin{proof}
Recall that $K \subseteq A \times B$ is a bisimulation, if the following two statements hold:
\begin{itemize}
\item whenever $(a, b) \in K$ and $(a, a') \in R$, there is a $b' \in B$ such that $(a', b') \in K$ and $(b, b') \in S$, and
\item whenever $(a, b) \in K$ and $(b, b') \in S$, there is an $a' \in A$ such that $(a', b') \in K$ and $(a, a') \in R$.
\end{itemize}
Therefore a bisimulation $K$ is nothing but a closed subset of $A \times B$ for the following non-deterministic inductive definition:
\begin{displaymath}
\begin{array}{ccc}
\infer[(a, a') \in R]{ \{ (a', b') \, : \, (b, b') \in S \} }{ \{ (a, b) \} } & &
\infer[(b, b') \in S]{ \{ (a', b') \, : \, (a, a') \in R \} }{ \{ (a, b) \} }
\end{array}
\end{displaymath}
Since this non-deterministic inductive definition consists of elementary rules, the statement of the proposition follows.
\end{proof}

\begin{prop}{minimalinNID}
If ${\cal R}$ is a (elementary, finitary) non-deterministic inductive definition on a set $X$, then (elementary, finitary) NID implies that there is a set ${\cal F}$ of ${\cal R}$-closed subsets of $X$ which is full: for every ${\cal R}$-closed subset $A$ of $X$, there is an element $F \in {\cal F}$ such that $A \subseteq F$. Hence NID implies that the minimal ${\cal R}$-closed subsets of $X$ form a set.
\end{prop}
\begin{proof}
Suppose ${\cal R}$ is a non-deterministic inductive definition on a set $X$. If * is an element not in $X$, then ${\cal R}$ can also be considered as a non-deterministic inductive definition on $X \cup \{ * \}$. Assume that ${\cal G}$ generates for this non-deterministic inductive definition on $X \cup \{ * \}$. Then let
\[ {\cal F} = \{ F - \{ * \} \, : \, F \in {\cal G}, * \in F \} \subseteq {\rm Clos}_{\cal R}(X). \]
The set ${\cal F}$ is full, because if $A$ is an ${\cal R}$-closed subset of $X$, then $A \cup \{ * \}$ is an ${\cal R}$-closed subset of $X \cup \{ * \}$. So there is an element $F \in {\cal G}$ with $* \in F$ and $F \subseteq A \cup \{ * \}$. Hence $A \supseteq F - \{ * \} \in {\cal F}$.

Because ${\cal F}$ is full,
\[ {\cal M} = \{ A \in {\cal F} \, : \, (\forall F \in {\cal F}) \, (F \subseteq A \Rightarrow F = A) \} \]
is the collection of all minimal ${\cal R}$-closed sets, which is a set by bounded separation.
\end{proof}

\begin{defi}{polynomial}
Let $f: B \to A$ be a function. The \emph{polynomial functor} $P_f$ associated to $f$ is defined on sets as
\[ P_f(X) = \{ (a \in A, t: f^{-1}(a) \to X) \}, \]
and on functions as
\[ P_f(\alpha)(a, t) = (a, \alpha t). \]

A \emph{$P_f$-algebra} is a set $X$ together with function $t: P_f(X) \to X$. A \emph{morphism of $P_f$-algebras} $ (X, s) \to (Y, t)$ is a function $\alpha: X \to Y$ such that the following diagram commutes:
\diag{ P_f(X) \ar[d]_s \ar[r]^{P_f(\alpha)} & P_f(Y) \ar[d]^t \\
X \ar[r]_{\alpha} & Y. }
The initial $P_f$-algebra, if it exists, is the \emph{W-type} associated to $f$.  The dual notions are that of a  \emph{$P_f$-coalgebra} and an \emph{M-type} associated to $f$. 
\end{defi}

In an impredicative metatheory such as {\bf IZF} one can prove the existence of W- and M-types. Indeed, if $f: B \to A$ is a function, then the M-type can be constructed as the collection of trees, with nodes labelled by elements $a \in A$ and edges labelled by elements $b \in B$, in such a way that $f^{-1}(a)$ enumerates the edges pointing towards a node labelled by $a \in A$. The following picture hopefully conveys the idea:
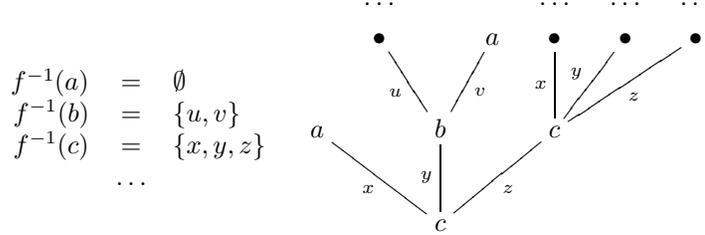
\begin{displaymath}
\xymatrix@C=.75pc@R=.5pc{ & & \ldots & & & \ldots & \ldots & \ldots \\
           & & {\bullet} \ar@{-}[dr]_u & & a \ar@{-}[dl]^v & {\bullet} \ar@{-}[d]_x & {\bullet} \ar@{-}[dl]_y & {\bullet} \ar@{-}[dll]^z \\
*{\begin{array}{rcl}
f^{-1}(a) & = & \emptyset \\
f^{-1}(b) & = & \{ u, v \} \\
f^{-1}(c) & = & \{ x, y, z \} \\
& \ldots & 
\end{array}}       & a \ar@{-}[drr]_x & & b \ar@{-}[d]_y & & c \ar@{-}[dll]^z & & \\
           & & & c & & & & } 
\end{displaymath}
If $M_f$ is this collection of trees, the coalgebra morphism $u: M_f \to P_f(M_f)$ takes a tree and sends it to the pair $(a, t)$, where $a$ is the label of the root of the tree and $t$ is the function sending an element $b \in f^{-1}(a)$ to the tree attached to the unique edge into the root with label $b$.

The W-type $W_f$ associated to $f$ consists of those trees in $M_f$ that are well-founded. In that case, the algebra morphism ${\rm sup}: P_f(W_f) \to W_f$ is the operation which takes an element $a \in A$ and a function $t: f^{-1}(a) \to W_f$ and creates the tree whose root is labelled by $a$ and with the tree $t(b)$ attached to the edge into the root with label $b \in f^{-1}(a)$. (For more on W- and M-types, see \cite{moerdijkpalmgren00, bergdemarchi07}.)

As said, both objects can be constructed inside {\bf IZF}: the M-type can be built by regarding trees as suitable collections of paths $\langle a_0, b_0, a_1, b_1, a_2, \ldots, a_n \rangle$ with $a_i \in A, b_i \in B$ and $f(b_i) = a_i$ for all $i \lt n$. The W-type can then be built by selecting the trees in the M-type that are well-founded, or as the least $P_f$-subalgebra of $u^{-1}: P_f(M_f) \to M_f$. As these constructions make use of the power set axiom, it is not at all clear whether M- and W-types can be shown to exist within {\bf CZF}. But with the NID principle we can, as we will now show.

\begin{theo}{elemNIDandMtypes}
Elementary NID implies that all M-types exist.
\end{theo}
\begin{proof}
We try to mimick the impredicative construction explained above.

Let $f: B \to A$ be a function and let $P$ be the collection of paths of odd length of the form $\langle a_0, b_0, a_1, b_1, a_2, \ldots, a_n \rangle$ such that
\begin{enumerate}
\item every $a_i$ belongs to $A$,
\item every $b_i$ belongs to $B$ and
\item for every $i \lt n$, we have $f(b_i) = a_i$.
\end{enumerate}

Consider the following non-deterministic inductive definition ${\cal R}$ on $P$:
\begin{displaymath}
\begin{array}{ccc}
\infer{ \{ \langle a \rangle \, : \, a \in A \} }{} & &
\infer[ b \in f^{-1}(a_n)]{ \{ \sigma * \langle b, a \rangle \, : a \in A \} }{ \{ \sigma = \langle a_0, b_0, \ldots, a_n \rangle \} }
\end{array}
\end{displaymath}
Let ${\cal M}$ be the collection of those ${\cal R}$-closed subsets $m \subseteq P$ such that
\begin{enumerate}
\item there is a unique $a \in A$ such that $\langle a \rangle \in m$,
\item $(\forall \sigma =\langle a_0, b_0, \ldots, a_n \rangle \in m) \, (\forall b \in f^{-1}(a_n)) \, (\exists ! a \in A) \, \sigma * \langle b, a \rangle \in m$,
\item $m$ is closed under initial segments.
\end{enumerate}
Elementary NID can used to justify the claim that ${\cal M}$ is a set, as follows. Note that every such $m \in {\cal M}$ must be minimal: if $A \subseteq m$ is also ${\cal R}$-closed, then one proves
\[ (\forall \sigma \in P) \, \sigma \in m \Rightarrow \sigma \in A \]
by induction on the length of $\sigma$. Elementary NID implies that the minimal ${\cal R}$-closed subsets of $X$ form a set (see \refprop{minimalinNID}), so ${\cal M}$ is a set by bounded separation.

Claim: ${\cal M}$ is a $P_f$-coalgebra. Proof: Let $m \in {\cal M}$. By 1, there is a unique element $a \in A$ such that $< a > \in A$. Then define for every $b \in B_a$,
\[ t(b) = \{ \sigma \in P \, : \, <a, b> * \sigma \in m \}. \]
Obviously, $t(b) \in {\cal M}$, so we have defined an operation $u: {\cal M} \to P_f({\cal M})$.

Claim: ${\cal M}$ is the final $P_f$-coalgebra. Proof: let $v: X \to P_f(X)$ be a $P_f$-coalgebra and let $v_1: X \to A$ be composition of $v$ with the projection on the first coordinate. By a path in $X$ we mean a sequence $\langle x_0, b_0, x_1, \ldots, x_n \rangle$ such that
\begin{enumerate}
\item every $x_i$ belongs to $X$,
\item every $b_i$ belongs to $B$ and
\item for every $i \lt n$, if $v(x_i) = (a_i, t_i)$, then $f(b_i) = a_i$ and $x_{i+1} = t(b_i)$.
\end{enumerate}
Then define $\hat{v}: X \to {\cal M}$ by saying that
\begin{eqnarray*}
\hat{v}(x) & = & \{ \langle a_0, b_1, a_1, \ldots, a_n \rangle \in P \, : \, \mbox{ there is a path } \langle x_0, b_1, x_1, \ldots, x_n \rangle \mbox{ in } X, \\
& & \mbox{ with } x = x_0 \mbox{ and } v_1(x_i) = a_i \mbox{ for all } i \lt n \}.
\end{eqnarray*}
This is clearly well-defined and the unique $P_f$-coalgebra morphism from $X$ to ${\cal M}$.
\end{proof}

Full NID can be used to prove that all W-types exist. This will follow from the previous result together with:

\begin{prop}{NIDimpliessetcompactness}
Let ${\cal R}$ be a deterministic inductive definition on a set $X$. Then NID implies that the ${\cal R}$-closed subsets of $X$ are set-generated and there is a least ${\cal R}$-closed subset of $X$.
\end{prop}
\begin{proof}
Suppose ${\cal R}$ is a \emph{deterministic} inductive definition on a set $X$. Clearly, NID implies that the ${\cal R}$-closed subsets of $X$ are set-generated. Suppose * is an element not in $X$ and write $X_* = X \cup \{ * \}$ and ${\cal R}_* = {\cal R} \cup \{ (\emptyset, \{*\} ) \}$. The ${\cal R}_*$-closed subsets of $X_*$ are closed under small intersections, so if ${\cal G}_*$ generates the ${\cal R}_*$-closed subsets of $X_*$, then $\bigcap {\cal G}_*$ is the least ${\cal R}_*$-closed subset of $X_*$. For if $A$ is ${\cal R}_*$-closed, then $* \in A$ and therefore there is a $G \in {\cal G}_*$ such that $* \in G$ and $G \subseteq A$. Then $\bigcap {\cal G}_* \subseteq G \subseteq A$. Now $M = \bigcap {\cal G}_* - \{ * \}$ is the least ${\cal R}$-closed subset of $X$. For if $A \subseteq X$ is ${\cal R}$-closed, then $A \cup \{ * \} \subseteq X_*$ is ${\cal R}_*$-closed, and therefore $\bigcap {\cal G}_* \subseteq A \cup \{ * \}$.
\end{proof}

\begin{theo}{NIDandWtypes}
NID implies that all W-types exist.
\end{theo}
\begin{proof}
Again, we try to mimick the impredicative construction explained above.

Let $f: B \to A$ be a function. We have shown in \reftheo{elemNIDandMtypes} that elementary NID implies that the corresponding M-type $u: {\cal M} \to P_f({\cal M})$ exists. By Lambek's lemma, $u$ is an isomorphism, so has an inverse, which we will call ${\rm sup}$. Then consider the following \emph{deterministic} inductive definition on ${\cal M}$:
\begin{displaymath}
\infer[a \in A, t: B_a \to {\cal M}]{ \{ {\rm sup}_a t \} }{ \{ t(b) \, : \, b \in B_a \} }
\end{displaymath}
By \refprop{NIDimpliessetcompactness}, this inductive definition has a least fixed point. This least fixed point is the W-type associated to $f$.
\end{proof}

\begin{rema}{onsetcomp}
We have shown that NID implies that for every (deterministic) inductive definition $\Phi$ on a set $X$ and every subset $A$ of $X$ the class $I(\Phi, A)$ (the least $\Phi$-closed subclass of $X$ containing $A$) is actually a set. What it does not seem to imply (at least, we failed to show that it does) is Aczel's Set Compactness Theorem \cite{aczelrathjen01}: this is the statement that there is a set $B$ of subsets of $X$ such that, whenever $a \in I(\Phi, A)$ there is a $Y \in B$ such that $Y \subseteq A$ and $a \in I(\Phi, Y)$.
\end{rema}

\section{NID and logic}

For developing the applications of the NID principle to formal topology it will be convenient to reformulate this principle using concepts from logic.

\subsection{Propositional case}

In this section we will identify models of a propositional theory (over a set of propositional letters $P$) with subsets of $P$ (where the elements that belong to the subset are precisely those which are true in the model).

\begin{defi}{gamelogic}
Let $P$ be a collection of propositional letters. A \emph{game formula} (over $P$) is a formula in propositional logic built using propositional letters from $P$ and infinitary disjunctions and conjunctions (but no implications and negations). A \emph{game sequent} (over $P$) is a formula of the form $\varphi \to \psi$ where $\varphi$ and $\varphi$ are game formulae. A \emph{game theory} (over $P$) is a set of game sequents.

An \emph{elementary game formula} is a game formula in which no infinite conjuctions occur. A game formula in which all conjunctions are finite is called a \emph{finitary game formula}. An \emph{elementary (finitary) game sequent} is a game sequent $\varphi \to \psi$ in which the hypothesis $\varphi$ is elementary (finitary). A collection of elementary (finitary) game sequents is called an \emph{elementary (finitary) game theory}.
\end{defi}

\begin{theo}{nidtogamelogic}
Full NID implies that class of models of an (elementary, finitary) game theory $T$ over a set of propositional letters $P$ is set-generated. The same statements holds for elementary (finitary) NID and models of elementary (finitary) game theories.
\end{theo}
\begin{proof}
Let $T$ be a game theory over a set of propositional letters $P$ and let $S$ be the union of $P$ with the collection of subformulae of $T$. Write $S'$ for the collection of all (or all elementary, or all finitary) game formulas in $S$. We consider the following non-deterministic inductive definition on $S$:
\begin{displaymath}
\begin{array}{ccc}
\infer[\bigwedge_{i \in I} \varphi_i \in S, i_0 \in I]{ \{ \varphi_{i_0} \} }{ \{ \bigwedge_{i \in I} \varphi_i \} } & &
\infer[\bigvee_{i \in I} \varphi_i \in S]{ \{ \varphi_i \, : \, I \in I \} }{ \{ \bigvee_{i \in I} \varphi_i \} } \\ \\
\infer[\bigwedge_{i \in I} \varphi_i \in S']{ \{ \bigwedge_{i \in I} \varphi_i \} }{ \{ \varphi_i \, : \, i \in I\} } & &
\infer[\bigvee_{i \in I} \varphi_i \in S, i_0 \in I]{ \{ \bigvee_{i \in I} \varphi_i \} }{ \{ \varphi_{i_0} \} } \\ \\
& \infer[\varphi \to \psi \in T]{ \{ \psi \} }{ \{ \varphi \}}
\end{array}
\end{displaymath}
Clearly, if $M$ is a model of $T$, then $\{ \varphi \in S \, : \, M \models \varphi \}$ is a closed subset of $S$.

Conversely, if $X$ is a closed subset of $S$, write $M = \{ p \in P \, : \, p \in X \}$. Now one proves by induction on the build-up of the game formula $\varphi$ that (1) if $\varphi \in S'$ and $M \models \varphi$, then $\varphi \in X$, and that (2) if $\varphi \in S$ and $\varphi \in X$, then $M \models \varphi$. So $M$ is model of $T$, because if $\varphi \to \psi \in T$, then
\[ M \models \varphi \Rightarrow \varphi \in X \Rightarrow \psi \in X \Rightarrow M \models \psi. \]

So if ${\cal G}$ is a generating set for the non-deterministic inductive definition, then $\{ \{ p \in P \, : \, p \in X \} : \, X \in {\cal G} \}$ generates the class of models of $T$.
\end{proof}

\begin{coro}{fullNIDprop1stcorr}
Full NID implies that the minimal models of a game theory $T$ over a set of propositional letters $P$ form a set. The same statements holds for elementary (finitary) NID and models of elementary (finitary) game theories.
\end{coro}

\begin{coro}{fullNIDprop2ndcorr}
The statement that for every set of propositional letters and game theory $T$ the class of models is set-generated is equivalent to full NID. The same statements holds for elementary (finitary) NID and models of elementary (finitary) game theories.
\end{coro}
\begin{proof}
In view of \reftheo{nidtogamelogic}, it suffices to show that the statement implies full NID. So suppose $X$ is a set and $\Phi \subseteq {\rm Pow}(X) \times {\rm Pow}(X)$ is a non-deterministic inductive definition and consider the following propositional theory over $X$:
\[ \{ \, \bigwedge \alpha \to \bigvee \beta \, : \, (\alpha, \beta) \in \Phi \, \}. \]
A model of this theory is the same as a $\Phi$-closed subset of $X$, so the result follows.
\end{proof}

\subsection{First-order case}

In this subsection we extend the notion of game sequent to first-order logic:

\begin{defi}{gamelogic1storder}
A \emph{game formula} (over some signature $\Sigma$) is a formula built from atomic formulas of the form $R(t_1, \ldots, t_n)$ with $R \in \Sigma$ (but no equalities) using infinitary disjunctions and conjunctions and existential and universal quantification (but no implications or negations). A \emph{game sequent} (over $\Sigma$) is the (universal closure) of a formula of the form $\varphi \to \psi$ where $\varphi$ and $\varphi$ are game formulae. A \emph{game theory} (over $\Sigma$) is a set of game sequents.
\end{defi}

Let $\Sigma$ be a signature and ${\cal R}$ a set of relation symbols not occuring in $\Sigma$. We will write $\Sigma'= \Sigma \cup {\cal R}$. Assume moreover that $M$ is a model in the signature $\Sigma$. As usual, we will regard $M'$ as a \emph{$\Sigma'$-expansion} of $M$, if $M'$ is a $\Sigma'$-model and $M' \upharpoonright \Sigma = M$ (where the latter means that $M$ and $M'$ have the same underlying set and the interpretation of the symbols belonging to $\Sigma$ in $M'$ coincides with their interpretation in $M$). If ${\cal M}$ is a collection of $\Sigma'$-expansions of $M$, we will call the $\Sigma'$-expansion $M_0$ with
\[ R^{M_0} = \bigcup_{M'\in {\cal M}} R^{M'} \]
for all $R \in {\cal R}$ the \emph{union} of the family ${\cal M}$.

\begin{theo}{fullNIDinpredcase}
Full NID is equivalent to the statement:
\begin{quote}
Suppose $\Sigma$ and $\Sigma'$ are two signatures as above, $M$ is a $\Sigma$-model and $T$ is a game theory over $\Sigma'$. Then there is a set ${\cal M}$ consisting of $\Sigma'$-expansions of $M$ which model $T$ such that any $\Sigma'$-expansion of $M$ modelling $T$ can be obtained as a union of elements from ${\cal M}$.
\end{quote}
\end{theo}
\begin{proof} (Sketch.)
We really only have to prove that full NID implies the statement, which we do by reducing it to the propositional case.

First of all, we introduce constants for all elements in $M$ to obtain an extension $\Sigma''$ of the signature $\Sigma'$. These constants we can then use to eliminate existential quantifiers in $T$ in favour of infinite disjunctions and universal quantifiers in favour of conjunctions. Writing $P$ for the collection of atomic sentences in the signature $\Sigma''$, we now have a propositional game theory $T'$ over $P$. Adding to $T'$ all atomic sentences true in $M$ we get a theory $T''$ whose models are really $\Sigma'$-structures $M'$ which model $T$ and are such that
\[ R^M \subseteq R^{M'}  \]
for all relation symbols $R$ in $\Sigma$. Then, if ${\cal M}$ generates the class of such models,
\[ {\cal M}'= \{ M_0 \in {\cal M} \, : \, R^M = R^{M_0} \mbox{ for all } R \in \Sigma \} \]
generates the class of $\Sigma'$-expansions modelling $T$.
\end{proof}

\begin{coro}{fullNIDcoro1storder}
Suppose $\Sigma$ and $\Sigma'$ are two signatures as above, $M$ is a $\Sigma$-model and $T$ is a game theory over $\Sigma'$. Call a $\Sigma'$-expansion $M'$ of $M$ \emph{minimal}, if for every other $\Sigma'$-expansion $M''$ on $M$ for which we have
\[ R^{M''} \subseteq R^{M'} \]
for all $R \in \Sigma'$, we actually have $R^{M''} = R^{M'}$ for all $R \in \Sigma'$. Then full NID implies that the collection of minimal $\Sigma'$-expansions of $M$ forms a set.
\end{coro}

\begin{exam}{linearorders}
To illustrate the usefulness of the last result, consider the following example. Let $(\mathbb{P}, \leq)$ be a partial order. Note that the structure of a linear order on $\mathbb{P}$ extending $\leq$ is the same thing as a $( \leq, \sim, \unlhd )$-expansion of $(\mathbb{P}, \leq)$, which models the set of finitary game sequents
\begin{eqnarray*}
& & p \sim p \\
p \sim q & \to & q \sim p \\
p \sim q \land q \sim r & \to & p \sim r \\
p \leq q & \to & p \unlhd q \\
& & p \unlhd p \\
p \unlhd q \land q \unlhd p & \to & p \sim q \\
p \unlhd q \land q \unlhd r & \to & p \unlhd r \\
p \sim q \land p' \sim q' \land p \unlhd q & \to & p' \unlhd q'
\end{eqnarray*}
and for which $\sim$ coincides with the equality on $\mathbb{P}$. Since every such model is automatically minimal (in the sense of the previous corollary) by linearity, we see that finitary NID implies that the collection of linear order structures on $\mathbb{P}$ extending $\leq$ forms a set.
\end{exam}

\section{Applications to formal topology}

In this section we will illustrate the power of non-deterministic inductive definitions by applying them to formal topology. This is not the place, however, to recap on the development of formal topology in the context of {\bf CZF} (for that, see \cite{aczel06} and \cite{bergmoerdijk10c}). Therefore this section will not be self-contained and will presuppose some familiarity with the basic notions of formal topology.

\reftheo{pointsandfinitaryNID} and \reftheo{elemNIDandcoeqalizers} below were originally proved by Erik Palmgren in the context of type theory (see \cite{palmgren05, palmgren06}). He used regular universes; we prove these results using elementary and finitary NID, showing that the regular universes can be avoided (see \refcoro{RDCimpliesSGA} below).

Recall that a \emph{point} of a formal space $(\mathbb{P},  \mbox{Cov})$ is an inhabited subset $\alpha \subseteq \mathbb{P}$ such that
\begin{enumerate}
\item[(1)] $\alpha$ is upwards closed,
\item[(2)] $\alpha$ is downwards directed,
\item[(3)] if $S \in \mbox{Cov}(a)$ and $a \in \alpha$, then $S \cap \alpha$ is inhabited.
\end{enumerate}

\begin{theo}{pointsandfinitaryNID}
Finitary NID implies that the collection of points of a set-presented formal space is set-generated.
\end{theo}
\begin{proof}
Suppose ${\rm BCov}$ is a presentation for the formal space $(\mathbb{P},  \mbox{Cov})$. In that case a point of the formal space $(\mathbb{P},  \mbox{Cov})$ is nothing but a closed set for the following non-deterministic inductive definition:
\begin{displaymath}
\begin{array}{ccc}
\infer[p \leq q]{ \{ q \} }{ \{ p \} } & \infer[p, q \in \mathbb{P}]{ \{ r \in \mathbb{P} \, : \, r \leq p, r \leq q \} }{ \{ p, q \} } &
\infer[ S \in {\rm BCov}(p)]{ \downarrow S }{ \{ a \} },
\end{array}
\end{displaymath}
where $\downarrow S = \{ r \in \mathbb{P} \, : \, (\exists s \in S) \, r \leq s \}$. Since this non-deterministic inductive definition is finitary, the result follows.
\end{proof}

Call a formal space \emph{flat}, if all its points are minimal with respect to the inclusion ordering (note that this is equivalent to saying that all its points are maximal). Sambin has shown that all regular formal spaces are flat (see \cite{sambin03}).

\begin{coro}{flatformsphavesetofpoints}
Finitary NID implies that flat, set-presented formal spaces have a set of points.
\end{coro}

\begin{defi}{morphismofformalspaces}
A \emph{continuous map} or a \emph{morphism of formal spaces} $F: (\mathbb{P}, {\rm Cov}) \to (\mathbb{Q}, {\rm Cov'})$ is a relation $F \subseteq \mathbb{P} \times \mathbb{Q}$ such that:
\begin{enumerate}
\item[(1)] If $F(p, q)$, $p' \leq p$ and $q \leq q'$, then $F(p', q')$.
\item[(2)] For every $q \in \mathbb{Q}$, the set $\{ p \, : \, F(p, q) \}$ is  closed under the covering relation.
\item[(3)] For every $p \in \mathbb{P}$ there is a a cover $S \in {\rm Cov}(p)$ such that each $p' \in S$ is related via $F$ to some element $q' \in \mathbb{Q}$.
\item[(4)] For every $q_0, q_1 \in \mathbb{Q}$ and element $p \in \mathbb{P}$ such that $F(p, q_0)$ and $F(p, q_1)$, there is a cover $S \in {\rm Cov}(p)$ such that every $p' \in S$ is related via $F$ to an element which is smaller than or equal to both $q_0$ and $q_1$.
\item[(5)] Whenever $F(p, q)$ and $T$ covers $q$, there is a sieve $S$ covering $p$, such that every $p' \in S$ is related via $F$ to some $q' \in T$.
\end{enumerate}
\end{defi}

\begin{theo}{elemNIDandcoeqalizers}
Elementary NID implies that the category of set-presented formal spaces has all coequalizers.
\end{theo}
\begin{proof}
See Proposition 7.9 in \cite{aczeletal12}: the key step amounts to showing that the class of models of a certain elementary game theory is set-generated.
\end{proof}

\begin{rema}{workofishiharaandkawai}
In \cite{ishiharakawai12}, Ishihara and Kawai use non-deterministic inductive definitions to show that coequalizers exist in the categories of basic pairs and concrete spaces as introduced by Sambin \cite{sambin03, sambin12}.
\end{rema}

The following result is new.

\begin{theo}{setgenerationformorphisms}
NID implies that the class of morphisms between two set-presented formal spaces is set-generated.
\end{theo}
\begin{proof}
Suppose ${\rm BCov}$ is a presentation for $(\mathbb{P}, {\rm Cov})$ and ${\rm BCov}'$ is a presentation for $(\mathbb{Q}, {\rm Cov'})$. Then a continuous morphism $F: (\mathbb{P}, {\rm Cov}) \to (\mathbb{Q}, {\rm Cov'})$ is nothing but a collection of propositional letters $\{ F(p, q) \, : \,  p \in \mathbb{P}, q \in \mathbb{Q} \}$ satisfying the following game sequents:
\begin{eqnarray*}
F(p, q) & \to & F(p', q') \quad \mbox{for all }  p' \leq p \mbox{ and } q \leq q'  \\
\bigwedge_{q' \in T} F(p, q') & \to & F(p, q) \quad \mbox{for all } T \in {\rm BCov}'(q) \\
& & \bigvee_{S \in {\rm BCov}(p)} \, \bigwedge_{p' \in S} \,  \bigvee_{q' \in \mathbb{Q}} F(p', q') \quad \mbox{for all } p \in \mathbb{P}  \\
F(p, q_0) \land F(p, q_1) & \to & \bigvee_{S \in {\rm BCov}(p)} \, \bigwedge_{p' \in S} \,  \bigvee_{q' \leq q_0, \, q_1} F(p', q') \\
F(p, q) & \to & \bigvee_{S \in {\rm BCov}(p)} \, \bigwedge_{p' \in S} \, \bigvee_{q' \in T} F(p', q') \quad \mbox{for all } T \in {\rm BCov}'(q)
\end{eqnarray*}
This shows the desired result.
\end{proof}

\section{Justification of the NID principle}

In order to show that the NID principle is acceptable from a  generalised-predicative perspective, we will show that is validated by Aczel's interpretation of {\bf CZF} in Martin-L\"of's type theory as in \cite{aczel78,aczel82,aczel86} (provided type theory comes equipped with an iterative universe closed under W-types). As the argument is rather complex, we will proceed in several steps. In the Section 6.2 we will give a first argument, inspired by Appendix A of \cite{aczel06}; it establishes slightly less than what we just claimed, because it requires a strong form of the regular extension axiom. In the Section 6.3 we will sharpen this argument to obtain the desired result, exploiting ideas that Ishihara used to derive finitary NID in {\bf CZF} + {\bf RDC} (see \refcoro{RDCimpliesSGA} below). But first we collect those properties of Aczel's interpretation that we will need for our proofs.

\subsection{Properties of Aczel's interpretation of CZF}

The crucial property of Aczel's interpretation that we will need is that it validates the Presentation Axiom {\bf PA} (see \cite{aczel82}). Recall that a \emph{base} (or a \emph{projective}) is a set $A$ such that every surjection $f: X \to A$ has a section. (It follows from the collection axiom that also every surjective map $f: X \to A$ from a class $X$ to a base $A$ is split.) The presentation axiom says that every set is the surjective image of a base.

The presentation axiom immediately has some interesting consequences. Call a map $g: D \to C$ a \emph{base map}, if both its codomain $C$ and all its fibres $D_c = f^{-1}(c)$ are bases.

\begin{lemm}{paandcollsquares}
The presentation axiom {\bf PA} implies that every map $g: B \to A$ fits into a commuting square
\diaglab{colsquare}{ D \ar[d]_h \ar@{->>}[r]^q & B \ar[d]^g \\
C \ar@{->>}[r]_p & A }
such that
\begin{enumerate}
\item $p$ is surjective,
\item the induced map $D \to C \times_A B$ is surjective and
\item $h$ is a base map.
\end{enumerate}
Moreover, every such square is a \emph{collection square}, in that for any $a \in A$ and any surjection $l: E \to B_a$ there is a $c \in C$ with $p(c) = a$ and a map $k: D_c \to E$ such that $l \circ k = q_c$.
\end{lemm}
\begin{proof}
We just outline the construction. First one applies {\bf PA} to cover $A$ with a base via a map $p: C \to A$. Applying {\bf PA} again we see that
\begin{quote}
for each $c \in C$ there exists a base $D$ and a cover $q: D \to B_{p(c)}$.
\end{quote}
Using the fact that $C$ is a base we find bases $D_c$ and covers $q_c: D_c \to B_{p(c)}$ as a function of $c \in C$. This completes the construction.
\end{proof}

A more refined analysis shows that the interpretation validates the principle that the class of bases is closed under exponentials (see \cite[Theorem 3.6]{aczel86}). This can be used to show that the following dependent choice principle for W-types (see \cite{palmgren05a,palmgren06}) is valid as well.

\begin{theo}{dependentchoiceforWtypes}
Let $f: B \to A$ be a base map. Then the interpretation of {\bf CZF} in the type theory {\bf ML$_{1W}$V} validates the following dependent choice principle for W-types:
\begin{quote}
If $X$ is a set and for every $a \in A$ there is a total relation
\[ R_a \subseteq X^{B_a} \times X, \]
then there is a function $h: W_f \to X$ such that for every ${\rm sup}_a t \in W_f$ one has $(h \circ t, h({\rm sup}_a t)) \in R_a$.
\end{quote}
It also validates the ``relativised'' version of this principle where $X$ can be a class.
\end{theo}
\begin{proof}
Cover $X$ with a base via a map $p: Y \to X$. Then we obtain for every $a \in A$ a total relation
\[ S_a \subseteq Y^{B_a} \times Y \]
defined by
\[ (s, y) \in S_a \Longleftrightarrow (p \circ s, p(y)) \in R_a. \]
Since $A$, $Y$ and the $B_a$ are bases and, under the interpretation, bases are closed under exponentials, we get for every $a \in A$ a function $\sigma_a: Y^{B_a} \to Y$ such that $S_a(s, \sigma_a(s))$ for all $s \in Y^{B_a}$. This gives $Y$ the structure of a $P_f$-algebra and hence we get a function $g: W_f \to Y$ such that $g({\rm sup}_a(t)) = \sigma_a(g \circ t)$.

Set $h = p \circ g$. For every ${\rm sup}_a(t) \in W_f$ we have that $S_a(g \circ t, \sigma_a(g \circ t))$, hence $R_a(p \circ g \circ t, p(\sigma_a(g \circ t)))$ and $R_a(h \circ t, h({\rm sup}_a(t)))$ as desired.

The principle is also validated if $X$ is a class, but here we only sketch the argument. First note that the relativised principle is valid in type theory (the proof of \cite[Theorem 7.2]{palmgren05a} never uses the fact that $X$ is a small type). Observe also that the interpretation of {\bf CZF} in {\bf ML$_{1W}$V} validates the statement that $W_f$ is injectively presented (see the Lemma on page 47 of \cite{aczel86}). Hence the statement follows as in \cite[Theorem 5.6]{aczel82}.
\end{proof}

\subsection{First proof}

The aim of this subsection is to prove:

\begin{theo}{1stproofoffullNID}
Full NID follows from the axiom of dependent choice for W-types, the presentation axiom and usREA.
\end{theo}

We need to define usREA.

\begin{defi}{usREA}
A set $U$ is \emph{regular} if it is \emph{transitive}, i.e., $a \in b \in U$ implies $a \in U$, and for each $a \in U$ and total relation $R$ from $a$ to $U$ there exists $b \in U$ such that
\[ (\forall x \in a) \, (\exists y \in b) \, R(x, y) \land (\forall y \in b) \, (\exists x \in a) \, R(x, y). \]
The set $U$ is \emph{union-closed}, if for every $x \in U$ also $\bigcup x \in U$. And $U$ is called \emph{separative}, if for any $a, b \in U$ also $\{ \emptyset \, : \, a \subseteq b \} \in U$.

The axiom usREA states: every set is a subset of a union-closed regular separative set.
\end{defi}

\begin{proof} (Of \reftheo{1stproofoffullNID}.)
Suppose ${\cal R}$ is a non-deterministic inductive definition on a set $X$.

First, let
\[ g: \sum_{(a, b) \in {\cal R}} a \to {\cal R} \]
be the first projection. Using \reflemm{paandcollsquares}, we find a base map $h$ and a collection square of the form:
\diaglab{2ndcolsquare}{ D \ar[d]_h \ar[r]^(.35)q & \sum_{(a, b) \in {\cal R}} a \ar[d]^g \\
C \ar[r]_p & {\cal R}. }
We will write ${\rm incl}_1: C \to C + 1$ for the inclusion into the first component, $f = {\rm incl}_1 \circ h: D \to C + 1$ and $W = W_f$.

Next, let $U$ be a set which is regular, union-closed and separative, and contains ${\cal R}$, $\{ x \}$ for all $x \in X$, $D_c$ for all $c \in C$ and $W_f$. 

\begin{lemm}{propofU}
Let $u \in U$.
\begin{enumerate}
\item If $t: u \to U$ is any map with $u \in U$, then ${\rm Im}(t) \in U$ and $\bigcup {\rm Im}(t) \in U$.
\item ${\cal R}_u = \{ (a, b) \in {\cal R} \, : \, a \subseteq u \} \in U$.
\end{enumerate}
\end{lemm}
\begin{proof}
The first point follows immediately from the fact that $U$ is regular and union-closed.

To show the second point, observe that if $a \in U$, then $\{ (a, b) \, : a \subseteq u \} \in U$, because $U$ is separative and regular. Hence we have a function $t: {\cal R} \to U$ which sends $(a, b) \in {\cal R}$ to $\{ (a, b) \, : a \subseteq u \}$. But then it follows from the first point that ${\cal R}_u = \bigcup {\rm Im}(t) \in U$.
\end{proof}

We claim that
\[ \Sigma = \{ u \in U \, : \, u \mbox{ is closed under } \cal{R} \} \]
generates ${\rm Clos}_{\cal R}(X)$. To show this, let $Y$ be ${\cal R}$-closed and $y \in Y$. Our aim to is to construct a $\sigma \in \Sigma$ such that $y \in \sigma$ and $\sigma \subseteq Y$. We will construct this set $\sigma$ just after equation (\ref{justafter}).

Write
\[ P = {\rm Pow}_U(Y) = \{ u \in U \, : \, u \subseteq Y \} \]
and
\[ {\cal T} = \{ (u, v) \, : \, (\forall (a, b) \in {\cal R}_u) \, b \between v \}. \]

\begin{lemm}{weirdlemma}
$(\forall u \in P) \, (\exists v \in P) \, (u, v) \in {\cal T}$.
\end{lemm}
\begin{proof}
Suppose $u \in P$, so $u \in U$ and $u \subseteq Y$. Since $Y$ is ${\cal R}$-closed, we have:
\[ (\forall (a, b) \in {\cal R}_u) \, (\exists \beta \in U) \, \beta \in b \cap Y. \]
Since ${\cal R}_u \in U$ and $U$ is regular, there is a $v \in U$ with $v \subseteq Y$ such that
\[ (\forall (a, b) \in {\cal R}_u) \, (\exists \beta \in v) \, \beta \in b \cap Y. \]
This proves the lemma.
\end{proof}

We are now ready to apply the axiom for dependent choice for W-types to $W_f$.
\begin{quote}
The set is $P$. $R_* \subseteq 1 \times P$ consists only of the pair $(*, \{ y \})$. For every $c \in C$, the relation $R_{c} \subseteq P^{D_c} \times P$ consists of those pairs $(\phi, u)$ such that $(\bigcup {\rm Im}(\phi), u) \in {\cal T}$.
\end{quote}
By the lemma we just proved all the relations here are total, so the axiom of dependent choice for W-types gives us a map $\phi: W_f \to P$ such that for every ${\rm sup}_c (t) \in W_f$ we have
\begin{equation} \label{justafter}
 (\phi \circ t, \phi({\rm sup}_ c (t))) \in R_c.
 \end{equation}
Let $\sigma := \bigcup {\rm Im}(\phi) \in U$. By construction, $\sigma \in U$, $y \in \sigma$ and $\sigma \subseteq Y$. So it remains to show that $\sigma \in \Sigma$, that is, that $\sigma$ is ${\cal R}$-closed.

Suppose $a \subseteq \sigma$ and $(a, b) \in {\cal R}$. We need to show that $\sigma \between b$. Our assumption $a \subseteq \sigma = \bigcup {\rm Im}(\phi)$ means that
\[ (\forall x \in a) \, (\exists w \in W_f) \, x \in \phi(w). \]
Then, because \refdiag{2ndcolsquare} is a collection square, we obtain a $c \in C$ and a map $t: D_c \to W_f$ such that
\[ (\forall d \in D_c) \, q(d) \in (\phi \circ t)(d), \]
and therefore $a \subseteq \bigcup {\rm Im}(\phi \circ t)$.

Also, for ${\rm sup}_c(t) \in W_f$  we have
\[ (\phi \circ t, \phi({\rm sup}_c(t)) \in R_{c}, \]
hence $(\bigcup {\rm Im}(\phi \circ t), \phi({\rm sup}_c(t)) \in {\cal T}$. But then, by definition of ${\cal T}$, we get $\phi({\rm sup}_c(t)) \between b$. Because $\phi({\rm sup}_c(t)) \subseteq \bigcup {\rm Im}(\phi) = \sigma$, the proof is finished.
\end{proof}

\subsection{Second proof}

The proof in the previous subsection establishes a result which is weaker than desired, because it relies on the existence of universes. In the present section we eliminate these and replace them in favour of the relativised dependent choice axiom for W-types. A first step towards this goal is isolating all the uses of the regular universe in one proposition. We continue to use the same notation. So $X$ is a set and ${\cal R}$ is a non-deterministic inductive definition on $X$. Also the maps $f, g$ and $h$ are as before.

\begin{lemm}{propertiesofU}
There is a set ${\cal P} \subseteq {\rm Pow}(X)$ such that:
\begin{enumerate}
\item ${\cal P}$ contains all singletons.
\item If $t: D_c \to {\cal P}$ is any map, then $\bigcup {\rm Im}(t) \in {\cal P}$.
\item If for some $u \in {\cal P}$ and $A \subseteq X$, we have $(\forall (a, b) \in {\cal R}_u) \, (\exists x \in A) \, x \in b$, then there is a $v \in {\cal P}$ with $v \subseteq A$ such that $(\forall (a, b) \in {\cal R}_u) \, (\exists x \in v) \, x \in b$.
\end{enumerate}
\end{lemm}
\begin{proof}
Let $U$ be a big set which is regular, union-closed and separative and which contains ${\cal R}$, $\{ x \}$ for all $x \in X$ and $D_c$ for all $c \in C$. Then it follows from \reflemm{propofU} that ${\cal P} = \{ u \in U \, : \, u \subseteq X \}$ has all the desired properties.
\end{proof}

To show that this collects all the uses of the universe that we need, we are now going to prove that \reflemm{propertiesofU} together with the axiom of dependent choice from W-types to sets is sufficient for establishing full NID.

So let
\begin{eqnarray*}
\Delta & = & \{ k: W_f \to {\cal P} \, : \, (\forall \, {\rm sup}_{c}(t) \in W_f) \, (k \circ t, k({\rm sup}_{c}(t))) \in R_c \} \\
\Sigma & = & \{ \bigcup {\rm Im}(k) \, : \, k \in \Delta \}
\end{eqnarray*}
where $(\phi, u) \in R_c$ iff $(\bigcup {\rm Im}(\phi), u) \in {\cal T}$ and ${\cal T} \subseteq {\cal P} \times {\cal P}$ consists of those pairs $(u, v)$ such that
\[ (\forall (a,b) \in {\cal R}_u) \, v \between b. \]
\begin{lemm}{sigmaclosed}
$\Sigma \subseteq {\rm Clos}_{\cal R}(X)$.
\end{lemm}
\begin{proof}
Suppose $\sigma \in \Sigma$, so $\sigma =\bigcup {\rm Im}(k)$ for some $k \in \Delta$, and $a \subseteq \sigma$ for some $(a, b) \in {\cal R}$. Hence
\[ (\forall x \in a) \, (\exists w \in W_f) \, x \in k(w). \]
By the collection square property, we obtain a $c \in C$ and a map $t: D_c \to W_f$ such that
\[ (\forall d \in D_c) \, q(d) \in (k \circ t)(d). \]
It follows that $a \subseteq \bigcup {\rm Im}(k \circ t)$. For $w = {\rm sup}_{c} t$ we have that $(k \circ t, k(w)) \in R_{c}$ and therefore $(\bigcup {\rm Im}(k \circ t), k(w)) \in {\cal T}$ and $k(w) \between b$. Since $k(w) \subseteq \sigma$, this finishes the proof.
\end{proof}

\begin{lemm}{sigmagenerating}
$\Sigma$ generates ${\rm Clos}_{\cal R}(X)$.
\end{lemm}
\begin{proof}
Suppose $Y \subseteq X$ is ${\cal R}$-closed and $y \in Y$. Let $P = \{ u \in {\cal P} \, : \, u \subseteq Y \}$. We first show that
\[ (\forall u \in P) \, (\exists v \in P) \, (u, v) \in {\cal T}. \]
So let $u \in P$, that is, $u \in {\cal P}$ and $u \subseteq Y$. Since $Y$ is ${\cal R}$-closed and $u \subseteq Y$, we have:
\[ (\forall (a, b) \in {\cal R}_u) \, (\exists x \in Y) \, x \in b. \]
Precisely for this quantifier combination, ${\cal P}$ satisfies a fullness property: so we obtain a $v \in {\cal P}$ with $v \subseteq Y$ such that:
\[ (\forall (a, b) \in {\cal R}_u) \, (\exists x \in v) \, x \in b. \]
This is precisely what we want.

We are now going to apply dependent choice for W-types.
\begin{quote}
The set is $P$. $R_* \subseteq 1 \times P$ consists just of the pair $(*, \{ y \})$. For every $c \in C$, we take the relation $R_c \subseteq P^{D_c} \times P$.
\end{quote}
What we proved above shows that all relations $R_c$ are total. The axiom of dependent choice for W-types now gives us a map $k: W_f \to P$ in $\Delta$. So $\sigma = \bigcup {\rm Im}(k) \in \Sigma$. Moreover, $y \in \sigma$ and $\sigma \subseteq Y$ by construction.
\end{proof}

We are now going to give a different proof of \reflemm{propertiesofU} using the axiom of relativised dependent choice for W-types as in \reftheo{dependentchoiceforWtypes}. Recall that \reflemm{propertiesofU} says:
\begin{lemm}{propertiesofU2ndtime}
There is a set ${\cal P} \subseteq {\rm Pow}(X)$ such that:
\begin{enumerate}
\item ${\cal P}$ contains all singletons.
\item If $t: D_c \to {\cal P}$ is any map, then $\bigcup {\rm Im}(t) \in {\cal P}$.
\item If for some $u \in {\cal P}$ and $A \subseteq X$, we have $(\forall (a, b) \in {\cal R}_u) \, (\exists x \in A) \, x \in b$, then there is a $v \in {\cal P}$ with $v \subseteq A$, such that $(\forall (a, b) \in {\cal R}_u) \, (\exists x \in v) \, x \in b$.
\end{enumerate}
\end{lemm}
We first apply \reflemm{paandcollsquares} to construct a base map $k$ and a  collection square
\diag{ F \ar[d]_k \ar[r]^r & D \ar[d]^h \\
E \ar[r] & C.}
Then we let $\phi: F + \{ 0 \} \to E + \{ 0, 1 \}$ be $k + l$ with $l : \{ 0 \} \to \{ 0, 1 \}$ the map sending 0 to 0. Write $P = {\rm Pow}({\rm Pow}(X))$ and define ${\cal D} \subseteq P \times P$ to consists of those pairs $(U, V) \in P^2$ such that
\begin{quote}
for every $u \in U$ and $A \subseteq X$, if we have $(\forall (a, b) \in {\cal R}_u) \, (\exists x \in A) \, x \in b$, then there is a $v \in V$ with $v \subseteq A$, such that $(\forall (a, b) \in {\cal R}_u) \, (\exists x \in v) \, x \in b$.
\end{quote}
We wish to apply the axiom of relativised dependent choice for W-types to obtain a map from $W_{\phi}$ to $P$.
\begin{quote}
The class is $P = {\rm Pow}({\rm Pow}(X))$. $R_1 \subseteq 1 \times P$ consists just of the pair $(*, \{ \{ x \} \, : \, x \in X\})$. $R_0 \subseteq P \times P$ is ${\cal D}$. And, finally, for every $e \in E$, we take the relation $R_e \subseteq P^{F_e} \times P$ to consist of all $(\phi, U)$ such that
\[ U = \{ \bigcup {\rm Im}(m: F_e \to {\rm Pow}(X)) \, : \, (\forall f \in F_e) \, m(f) \in \phi(f) \}. \]
\end{quote}
We first need to check that all relations here are total.
\begin{lemm}{Dtotal}
$(\forall U \in P) \, (\exists V \in P) \, (U, V) \in {\cal D}$.
\end{lemm}
\begin{proof}
Fullness gives us, for every $u \in U$, a set $V$ such that for all $A \subseteq X$, if we have $(\forall (a, b) \in {\cal R}_u) \, (\exists x \in A) \, x \in b$, then there is a $v \in V$ with $v \subseteq A$, such that $(\forall (a, b) \in {\cal R}_u) \, (\exists x \in v) \, x \in b$. Applying the collection axiom to this statement and $U$ yields the desired conclusion.
\end{proof}

So the relativised dependent choice axiom for W-types gives us a map $\rho: W_{\phi} \to {\rm Pow}({\rm Pow}(X))$. Let ${\cal P} = \bigcup {\rm Im}(\rho)$. ${\cal P}$ clearly satisfies properties 1 and 3, hence we only check property 2.

\begin{lemm}{checkprop2}
${\cal P}$ satisfies property 2.
\end{lemm}
\begin{proof}
Suppose $s: D_c \to {\cal P}$ is any map. Then
\[ (\forall d \in D_c) \, (\exists w \in W_{\phi}) \, s(d) \in \rho(w). \]
By the collection square property, there is an $e \in E$ together with a map $t: F_e \to W_{\phi}$ such that
\[ (\forall f \in F_e) \, (s \circ r)(f) \in (\rho \circ t)(f). \]
Then $\bigcup {\rm Im}(s) = \bigcup {\rm Im}(s \circ r) \in {\cal P}$.
\end{proof}

We conclude:

\begin{theo}{2ndproofoffullNID}
Full NID follows from the presentation axiom and the axiom of relativised dependent choice for W-types.
\end{theo}

\begin{coro}{fullNIDvalidontypethint}
The NID principle is valid on the type-theoretic interpretation of {\bf CZF} in {\bf ML}$_{1W}${\bf V}.
\end{coro}

\section{Comparison with related work}

Originally this section was devoted to comparing the NID principle to some ideas developed by Peter Aczel and Hajime Ishihara in slides for talks and unpublished notes. However, after the author submitted this paper, their ideas coalesced in the SGA principle and the preprint \cite{aczeletal12}, written together with Takako Nemoto and Yasushi Sangu. So now it makes more sense to relate our work with what happens in \cite{aczeletal12}.

As it turns out, the relationship is very close: their SGA principle is equivalent to finitary NID. This section will be devoted to a proof of this fact.

\begin{defi}{setgenerated}
Let $S$ be a set and $X$ be a subclass of ${\rm Pow}(S)$. Recall that $X$ is \emph{set-generated}, if there is a set $G \subseteq X$ such that
\[ (\forall \alpha \in X) \, (\forall s \in \alpha) \, (\exists \beta \in G) \, s \in \beta \subseteq \alpha. \]
$X$ will be called \emph{strongly set-generated}, if there is a set $G \subseteq X$ such that
\[ (\forall \alpha \in X) \, (\sigma \in {\rm Finpow}(\alpha)) \, (\exists \beta \in G) \, \sigma \subseteq \beta \subseteq \alpha. \]
Here ${\rm Finpow}(\alpha)$ is the set of finite subsets of $\alpha$ (see \refrema{notation}).
\end{defi}

\begin{defi}{SGA}
The \emph{set generated axiom} (abbreviated as \emph{SGA}) is a principle which says that for each set $S$ and each subset $Z$ of ${\rm Finpow}(S) \times {\rm Pow}({\rm Pow}(S))$, the class
\[ {\cal M}(Z) := \{ \alpha \in {\rm Pow}(S) \, : \, (\forall (\sigma, \Gamma) \in Z) \, \sigma \subseteq \alpha \to (\exists U \in \Gamma) \, U \subseteq \alpha \} \]
is strongly set-generated.
\end{defi}

\begin{theo}{equivalent}
SGA and finitary NID are equivalent in {\bf CZF}.
\end{theo}
\begin{proof}
It is easy to see that SGA implies finitary NID: for if ${\cal R}$ is a set of finitary, non-deterministic rules on a set $S$, we put
\[ Z := \{ (a, \{ \{x \} \, : \, x \in b \}) \, : \, (a, b) \in {\cal R} \}. \]
Then ${\rm Clos}_{\cal R}(S) = {\cal M}(Z)$, so SGA implies that ${\rm Clos}_{\cal R}(S)$ is set-generated.

We now prove the converse. Let $S$ be a set and $Z \subseteq {\rm Finpow}(S) \times {\rm Pow}({\rm Pow}(S))$. Write
\begin{eqnarray*}
S^* & := & \bigcup \{ \Gamma \, : \, (\sigma, \Gamma) \in Z \} \cup {\rm Finpow}(S) \subseteq {\rm Pow}(S)
\end{eqnarray*}
(which is a set by the union and replacement axioms)
and consider the following finitary, non-deterministic definition ${\cal R}$ on $S^*$:
\begin{displaymath}
\begin{array}{ccccccc}
\infer[(\sigma, \Gamma) \in Z]{ \Gamma }{ \{ \sigma \}} &  &
\infer[u \in U \in S^*]{ \{ \{ u \} \} }{ \{ U \} } & &
\infer[\sigma \in {\rm Finpow}(S)]{ \{ \sigma \} }{ \{ \{ s \} \, : \, s \in \sigma \} }
\end{array}
\end{displaymath}
By finitary NID, there is a set $G^*$ which generates ${\rm Clos}_{\cal R}(S^*)$. Put
\[ G := \{ \{ s \in S \, : \, \{ s \} \in \gamma \} \, : \gamma \in G^* \}. \]

We first prove $G \subseteq {\cal M}(Z)$. So suppose $\alpha = \{ \{ s \in S \, : \, \{s \} \in \gamma \}$ with $\gamma \in G^*$, and suppose $(\sigma, \Gamma) \in Z$ with $\sigma \subseteq \alpha$. This implies that for every $s \in \sigma$ we have  $\{ s \} \in \gamma$. Hence we have $\sigma \in \gamma$ by applying the third rule. Then, by the first rule, we know that there is a set $U \in \gamma$ with $U \in \Gamma$. So it follows by the second rule that for every $u \in U$ we have $\{ u \} \in \gamma$, whence $u \in \alpha$. So $U \subseteq \alpha$, as desired.

Now we prove that $G$ strongly generates ${\cal M}(Z)$. Suppose $\alpha \in {\cal M}(Z)$ and $\sigma \in {\rm Finpow}(\alpha)$. We need to find a $\beta \in G$ with $\sigma \subseteq \beta \subseteq \alpha$. To this purpose, consider
\[ \gamma := S^* \cap {\rm Pow}(\alpha) \]
(which is a set by bounded separation). It is easy to see that $\gamma$ belongs to ${\rm Clos}_{\cal R}(S^*)$. So because $G^*$ generates, there is an element $\delta \in G^*$ with $\delta \subseteq \gamma$ and $\sigma \in \delta$. Write
\[ \beta := \{ s \in S \, : \{ s \} \in \delta \}. \]
By construction, $\beta \in G$. Moreover, $\sigma \subseteq \beta$, for if $s \in \sigma$, then we have $\{ s \} \in \delta$, because $\sigma \in \delta$ and $\delta$ is closed under the second rule; therefore $s \in \beta$ by definition. And, finally, we have $\beta \subseteq \alpha$, because if $s \in \beta$, then $\{ s \} \in \delta \subseteq \gamma \subseteq {\rm Pow}(\alpha)$; hence $s \in \alpha$.
\end{proof}

\begin{coro}{RDCimpliesSGA}
The axiom of relativised dependent choice {\bf RDC} implies finitary NID. Hence finitary NID is validated on the type-theoretic interpretation of {\bf CZF} in {\bf ML}$_1${\bf V}.
\end{coro}
\begin{proof}
The first statement follows from the previous result in combination with Theorem 5.1 in \cite{aczeletal12}. The second statement follows from the fact that {\bf RDC} is validated on the type-theoretic interpretation of {\bf CZF} in {\bf ML}$_1${\bf V} (see \cite{aczel82} and \cite{rathjentupailo06}).
\end{proof}

\begin{rema}{prooftheoreticstrength}
Since {\bf CZF} and {\bf ML}$_1${\bf V} have the same proof-theoretic strength (see \cite{grifforrathjen94}), it follows that finitary NID does not increase the proof-theoretic strength of {\bf CZF}.

On the other hand, it seems plausible that {\bf CZF} extended with the statement that all W-types exist has the same proof-theoretic strength as {\bf ML}$_{1W}${\bf V} and {\bf CZF} + {\bf REA}. If that is true, then it would follow from \reftheo{NIDandWtypes} and \refcoro{fullNIDvalidontypethint} that {\bf CZF} + {\bf NID} also has this strength; consequently, {\bf CZF} + {\bf NID} would proof-theoretically be a much stronger system than {\bf CZF}, and finitary NID would not imply full NID.
\end{rema}

\section{Conclusion and open questions}

We have introduced a new proof principle, the NID principle, and shown how it can be used to obtain results in the context of the constructive set theory {\bf CZF}, especially in formal topology. We are convinced that these results cannot be obtained in {\bf CZF} extended with either {\bf REA} or a combination of {\bf AMC} and {\bf WS}, but we do not have a proof of this fact.

We also believe that elementary NID cannot be proved in {\bf CZF} and that elementary NID does not imply finitary NID, but also here we lack proofs. Another question which we have left open is whether the NID principle is stable under such constructions from algebraic set theory as exact completion, realizability and sheaves. Again, this seems to us very likely to be true, but we have not tried very hard to find proofs.

\bibliographystyle{plain} \bibliography{ast}

\end{document}